\documentclass{siamltex}
\usepackage{amsfonts,latexsym,amsbsy,mathtools}
\usepackage{color}
\usepackage{graphicx}

\def\Am#1{A_{#1}}
\def\Aml#1#2{A_{#1}^{(#2)}}
\def\Vl#1#2{V_{#1,#2}^{}}
\def\Vlt#1#2{V_{#1,#2}^*}
\def\Hl#1{H_{#1 +1,#1}^{}}
\def\Hld#1{H_{#1 +1,#1}^{\dagger}}
\def\Hlt#1{H_{#1 +1,#1}^{*}}
\def\Pm#1{P_{_{#1}}}
\def\Qm#1{Q_{#1}}
\def\Rm#1{R_{#1}}
\def\Ul#1{U_{#1}^{}}
\def\Ult#1{U_{#1}^*}
\def\Wl#1{W_{#1}^{}}
\def\Wlt#1{W_{#1}^*}
\def\Sl#1{\Sigma_{#1+1,#1}^{}}
\def\Slt#1{\Sigma_{#1+1,#1}^*}
\def\Sld#1{\Sigma_{#1+1,#1}^\dagger}
\def\Ll#1{\Lambda_{#1+1}^{}}
\def\Ikl#1#2{I_{#1,#2}}
\def\Il#1{I_{#1}}
\def\cng#1{{\color{blue}#1}}

\newcommand\smallO{
  \mathchoice
    {{\scriptstyle\mathcal{O}}}
    {{\scriptstyle\mathcal{O}}}
    {{\scriptscriptstyle\mathcal{O}}}
    {\scalebox{.7}{$\scriptscriptstyle\mathcal{O}$}}
  }

\newtheorem{remark}[theorem]{Remark}

\usepackage{ifthen}
\usepackage{float}
\floatstyle{ruled}
\newfloat{algorithm}{OTB}{alg}
\floatname{algorithm}{Algorithm}
\newcounter{algo@row}
\newcounter{algo@rowindent}
\newcommand{\algofont}[1]{\textbf{#1}}
\newcommand{\algonumbersize}[1]{\scriptsize{#1}}
\newcommand{\algopreitem}[1][\arabic{algo@row}]{\texttt{\algonumbersize{#1}}}
\newcommand{\algoitemskip}{\hspace{\value{algo@rowindent}cc}}
\newenvironment{algo}{\vskip.3em\small%
  \begin{list}{\algopreitem\texttt{\algonumbersize{:}}}{%
      \usecounter{algo@row}%
      \setcounter{algo@rowindent}{0}%
      \setlength{\itemindent}{2em}%
      \setlength{\labelwidth}{2em}
      \setlength{\parsep}{0cm}%
    }%
}{
  \end{list}\vskip-.5em
}
\newcommand{\algonewnestedopen}[2]{
  \newcommand{#1}[1][]{%
    \ifthenelse{\equal{##1}{}}{\item}{\item[{\algopreitem[##1]}]}
    \algoitemskip\algofont{#2}%
    \addtocounter{algo@rowindent}{1}%
    \ignorespaces
  }
}
\newcommand{\algonewnestedaux}[2]{
  \newcommand{#1}[1][]{
    \addtocounter{algo@rowindent}{-1}
    \ifthenelse{\equal{##1}{}}{\item}{\item[{\algopreitem[##1]}]}
    \algoitemskip\algofont{#2}%
    \addtocounter{algo@rowindent}{+1}%
    \ignorespaces
  }
}
\newcommand{\algonewnestedclose}[2]{
  \newcommand{#1}[1][]{
    \addtocounter{algo@rowindent}{-1}
    \ifthenelse{\equal{##1}{}}{\item}{\item[{\algopreitem[##1]}]}
    \algoitemskip\algofont{#2}%
    \ignorespaces
  }
}
\newcommand{\algonewcommand}[2]{
  \newcommand{#1}[1][default]{
    \ifthenelse{\equal{##1}{default}}{\item}{\item[{\algopreitem[##1]}]}%
    \algoitemskip\algofont{#2}%
    \ignorespaces
  }%
}
\newcommand{\algonewkeyword}[2]{\newcommand{#1}{\algofont{#2}}}
\algonewcommand{\STATE}{\ignorespaces}
\algonewcommand{\INPUT}{Input: }
\algonewcommand{\pINPUT}{\phantom{Input: }}
\algonewcommand{\COMPUTE}{Compute: }
\algonewcommand{\OUTPUT}{Output: }
\algonewcommand{\pOUTPUT}{\phantom{Output: }}
\algonewnestedopen{\IF}{if }
\algonewnestedaux{\ELSEIF}{else if }
\algonewnestedaux{\ELSE}{else }
\algonewnestedclose{\ENDIF}{end if }
\algonewnestedopen{\FOR}{for }
\algonewnestedclose{\ENDFOR}{end for }
\algonewnestedopen{\WHILE}{while }
\algonewnestedclose{\ENDWHILE}{end while }
\algonewcommand{\BREAK}{break}%
\algonewkeyword{\For}{for }%
\algonewkeyword{\To}{to }%
\algonewkeyword{\Do}{do }%
\algonewkeyword{\If}{if }%
\algonewkeyword{\Then}{then }%
\algonewkeyword{\Else}{else }%
\algonewkeyword{\End}{end }%
\algonewkeyword{\AND}{and }%
\algonewkeyword{\True}{true }%
\algonewkeyword{\False}{false }%
\algonewkeyword{\Call}{call }%
\algonewkeyword{\irbleigs}{irbleigs }%
\algonewkeyword{\tridiag}{tridiag}%
\algonewkeyword{\reorth}{reorth}%
\begin{document}
\title{Error Estimates for Arnoldi--Tikhonov \\ Regularization for Ill-Posed
Operator Equations}
\author{Ronny Ramlau\thanks{Institute for Industrial Mathematics, Kepler University, Linz, and Johann Radon Institute for Computational and
Applied Mathematics, Austrian Academy of Sciences, Altenbergerstr. 69,
A-4040 Linz, Austria. E-mail: {\tt ronny.ramlau@oeaw.ac.at}}
\and
Lothar Reichel\thanks{Department of Mathematical Sciences, Kent State
University, Kent, OH 44242, USA. E-mail: {\tt reichel@math.kent.edu}}
}

\maketitle

\begin{abstract}
Most of the literature on the solution of linear ill-posed operator equations, or their
discretization, focuses only on the infinite-dimensional setting or only on the solution
of the algebraic linear system of equations obtained by discretization. This paper
discusses the influence of the discretization error on the computed solution. We consider
the situation when the discretization used yields an algebraic linear system of equations
with a large matrix. An approximate solution of this system is computed by first
determining a reduced system of fairly small size by carrying out a few steps of the Arnoldi
process. Tikhonov regularization is applied to the reduced problem and the regularization
parameter is determined by the discrepancy principle. Errors incurred in each step of the
solution process are discussed. Computed examples illustrate the error bounds derived.
\end{abstract}


\section{Introduction}
Let $A:{\mathcal X}\to{\mathcal Y}$ be an injective linear operator between the Hilbert spaces
${\mathcal X}$ and ${\mathcal Y}$, and assume that $A$ is not continuously invertible. We
are concerned with the solution of operator equations of the form
\begin{equation}\label{Opgl}
Ax=y,\qquad x\in{\mathcal X},\quad y\in{\mathcal Y}.
\end{equation}
Let $\|\cdot\|_{\mathcal X}$ and $\|\cdot\|_{\mathcal Y}$ denote the norms of the
spaces ${\mathcal X}$ and ${\mathcal Y}$, respectively. We will assume that equation
\eqref{Opgl} is consistent and are interested in determining the solution of minimal norm.
We denote this solution by $\widehat{x}$. The solution $\widehat{x}$ might not depend
continuously on $y$. Therefore its computation is an ill-posed problem.

The right-hand side $y$ of \eqref{Opgl} is assumed not to be available; only an
error-contaminated approximation $y^\delta\in{\mathcal Y}$ of $y$ is known. We assume that
$y^\delta$ satisfies
\begin{equation}\label{noise}
\|y-y^\delta\|_{\mathcal Y}\le \delta,
\end{equation}
with a known bound $\delta>0$. The solution of the equation
\begin{equation}\label{Opgldelta}
Ax=y^\delta,\qquad x\in{\mathcal X},\quad y^\delta\in{\mathcal Y},
\end{equation}
obtained by replacing $y$ by $y^\delta$ in \eqref{Opgl}, if it exists, generally, is not a
meaningful approximation of the desired solution $\widehat{x}$ since $A$ is not continuously
invertible. In fact, equation \eqref{Opgldelta} might not have a solution even when
equation \eqref{Opgl} does. A regularization method, which replaces the operator $A$ by a
nearby operator, such that the solution of the modified equation so obtained exists and is
less sensitive to the error in $y^\delta$, has to be used to obtain a meaningful
approximation of $\widehat{x}$.

The numerical solution of \eqref{Opgldelta} requires discretization at a certain stage of
the process. In general, this can be done in two ways:
\vskip3pt
\begin{enumerate}
\item[(i)] {\it Regularize then discretize:} In this approach, the infinite-dimensional
ill-posed problem is transformed into a well-posed problem, e.g., by means of {\it Tikhonov
regularization}. Then the well-known error estimates for the regularized solution, see,
e.g., \cite{EHN} for regularization in Hilbert spaces, can be applied. In a second step
the now well-posed regularized equation is discretized, and available error estimators for
well posed-problems, e.g., from the theory of finite elements, can be used. This approach
has been followed in, e.g., \cite{NEUBAUER1988507,MPRS2001,HHA,CKW}.
\vskip3pt
\item[(ii)] {\it Discretize then regularize:} The discretization of the ill-posed operator
equation \eqref{Opgldelta} yields a linear system of algebraic equations
\begin{equation}\label{discOpgl}
	A_{n}x_n=y_n^\delta
\end{equation}
with an ill-conditioned, possibly singular, matrix $A_{n}\in \mathbb{R}^{n,n}$, and
vectors $x_n,y_n^\delta\in\mathbb{R}^n$. Well-known methods from linear algebra are used
for its solution; see, e.g., \cite{CMRS,HH,Hansenbook}. The difficulty with this approach
is to obtain convergence and convergence rate results for the distance between the
solution of the infinite-dimensional problem \eqref{Opgl} and its discretized counterpart
$x_n$ fulfilling (\ref{discOpgl}), see, e.g., \cite{Natterer1977,DickenMaass1996}.
\end{enumerate}
\vskip3pt
As mentioned above, the approach (i) works particularly well for different variants of
Tikhonov regularization. Iterative methods, however, require frequent application of the
operator, and maybe of its adjoint, which is only possible in a discretized form. Iterative
methods therefore belong to category (ii). An analysis of approach (ii) has been carried
out for an adaptive version of Landweber iteration \cite{RamlauLandweber}, but to the best
of our knowledge this approach has not been investigated for Krylov subspace methods.
Additionally, methods that work exceptionally well in finite dimensions but have no
infinite-dimensional counterpart, or for which an error analysis is missing in infinite
dimensions, belong to category (ii). The latter holds for the method discussed in the
present paper.

In this paper, we start with two continuous linear operator equations, \eqref{Opgl} and
\eqref{Opgldelta}, and discretize the latter to obtain the linear system of algebraic
equations \eqref{discOpgl}. We are concerned with the situation when the matrix $A_n$ is
large and, in particular, when $A_n$ is too large to make the computation of its singular
value decomposition attractive. Then we apply the Arnoldi process to compute an
approximation of fairly low rank of the matrix $A_n$ in \eqref{discOpgl}. We replace $A_n$
in \eqref{discOpgl} by this low-rank approximation, and compute an approximate solution of
the linear system of equations with the aid of Tikhonov regularization. The replacement of
$A_n$ by a low-rank approximation reduces the computational effort required for Tikhonov
regularization when the matrix $A_n$ is large, which is the situation of interest to us.
Our approach allows us to solve problems with a matrix $A_n$ that is too large to make
the use of direct solution methods, which require factorization of a large matrix, e.g., of
$A_n$ or a related matrix, too expensive to be attractive or feasible. We will discuss the
effect on the computed solution of discretization errors that stem from replacing the
operator $A$ by the matrix $A_n$, as well as the effect of the error in the right-hand side
$y^\delta$. Moreover, we are concerned with the influence on the computed solution of the
replacement of the matrix $A_n$ in the linear system \eqref{discOpgl} by a low-rank matrix
determined by the Arnoldi process. We remark that Tikhonov regularization based on partial
Arnoldi decomposition, and some variations thereof, have been described in
\cite{CMRS,DR,GNR2,HMR,LR} and in references therein. The contribution of this paper is to
provide an error analysis.

This paper is organized as follows. Section \ref{sec2} discusses results by Natterer
\cite{Natterer1977} on the discretization of integral operators. Discretization
yields the linear system of algebraic equations \eqref{discOpgl}. We assume that the
matrix $A_n$ determined by discretization is so large that factorization is
unattractive or unfeasible. Section \ref{sec3} reviews the Arnoldi process for
computing an approximation of fairly low rank of the matrix $A_n$ in \eqref{discOpgl}. We
use this low-rank approximation in Tikhonov regularization and obtain a quite efficient
solution method. To analyze the performance of this solution approach, we have to take into
account the discretization error as well as the error incurred by approximating the matrix
$A_n$ by the low-rank approximation furnished by the Arnoldi process. Section \ref{sec4}
applies bounds due to Neubauer \cite{NEUBAUER1988507} to the computed solution obtained
by the Tikhonov regularized problem that uses the approximation of the matrix $A_n$ in
\eqref{discOpgl} computed with the Arnoldi process. We remark that while the bounds
provided by Natterer \cite{Natterer1977} shed light on the influence of the discretization
error on the computed solution, they are not useful for assessing the effect of
approximating the matrix $A_n$ by a low-rank approximation determined by the
Arnoldi process. We will comment further on this issue in Remark \ref{rm4.10} of
Section \ref{sec4}. A few
computed examples that illustrate the theory are presented in Section \ref{sec5}, and
concluding remarks can be found in Section \ref{sec6}.

\section{Discretization of the operator equation}\label{sec2}
To be able to numerically compute an approximate solution of equation \eqref{Opgldelta} in
the infinite-dimensional Hilbert space ${\mathcal X}$, the equation first has to be
discretized. This results in the finite-dimensional equation \eqref{discOpgl}. We
introduce a discretization and define a finite-dimensional least-squares problem similarly
as Natterer \cite{Natterer1977}, who investigated regularization properties of projection
methods.

Introduce the finite-dimensional subspaces
\begin{eqnarray*}
\mathcal{X}_n&\subset & {\mathcal X},\hspace{1cm} \dim ({\mathcal X}_n)=n,\\
{\mathcal Y}_n &=& A{\mathcal X}_n,
\end{eqnarray*}
and define projectors $P_n:{\mathcal X}\to{\mathcal X}_n$ and
$Q_n:{\mathcal Y}\to{\mathcal Y}_n$. The space $\mathcal{X}_n$ is chosen for its
convenience to use in applications and for the approximation properties of its elements.
For instance, $\mathcal{X}_n$ may be a space of piece-wise polynomials or finite
elements.

Consider the linear system of equations
\begin{equation}\label{linsys}
Q_nAP_n x = Q_n y^\delta.
\end{equation}
We identify the matrix $A_n$ and vector $y_n^\delta$ in \eqref{discOpgl} with the
finite-dimensional operator $Q_nAP_n$ and the right-hand side $Q_n y^\delta$ in
\eqref{linsys}, respectively. The unique least-squares solution of minimal norm of
equation
\eqref{linsys} is given by $x_n:=A_n^\dag y_n^\delta$, where $A_n^\dag$ denotes the
Moore--Penrose pseudoinverse of the matrix $A_n$. We identify this solution of
\eqref{linsys} with the solution in ${\mathbb R}^n$ of \eqref{discOpgl}.

Let $\{e_j\}_{j=1}^n$ form a convenient basis for ${\mathcal X}_n$, such as a
basis of piece-wise polynomials or finite elements with local support. Consider the
representation
\begin{equation}\label{nbasis}
x_n=\sum_{j=1}^n x_j^{(n)} e_j
\end{equation}
of an element $x_n\in{\mathcal X}_n$. We identify the function $x_n$ with the vector
\[
\vec{x}_n=(x_1^{(n)},x_2^{(n)},\ldots,x_n^{(n)})^T\in{\mathbb R}^n.
\]
To shed light on how $\|x_n\|_{\mathcal X}$ relates to $\|\vec{x}_n\|_2$, we
introduce an orthonormal basis $\{\widehat{e}_j\}_{j=1}^n$ for ${\mathcal X}_n$. There is
a nonsingular matrix $M_n=[m_{ij}]\in{\mathbb R}^{n\times n}$ such that
\[
(e_1,e_2,\ldots,e_n)=(\widehat{e}_1,\widehat{e}_2,\ldots,\widehat{e}_n)M_n,
\]
i.e., $e_j=\sum_{i=1}^n m_{i,j}\widehat{e}_i$ for $j=1,2,\ldots,n$.
For instance, when the basis $\{\widehat{e}_j\}_{j=1}^n$ is determined from
$\{e_j\}_{j=1}^n$ by the Gram--Schmidt process, the matrix $M_n$ is upper triangular.

We obtain from \eqref{nbasis} that
\[
x_n=(e_1,e_2,\ldots,e_n)\vec{x}_n=
(\widehat{e}_1,\widehat{e}_2,\ldots,\widehat{e}_n)M_n\vec{x}_n.
\]
It follows that
\begin{equation}\label{ubd}
\|x_n\|_{\mathcal X}=\|M_n\vec{x}_n\|_2\leq\|M_n\|_2\|\vec{x}_n\|_2=
\sigma_{\max}(M_n)\|\vec{x}_n\|_2,
\end{equation}
where $\sigma_{\max}(M_n)$ denotes the largest singular value of the matrix $M_n$. Let
$\sigma_{\min}(M_n)$ stand for the smallest singular value of $M_n$. Then we obtain
analogously to \eqref{ubd} that
\[
\|x_n\|_{\mathcal X}\geq\sigma_{\min}(M_n)\|\vec{x}_n\|_2.
\]
We will assume that there are constants $c_{\min}$ and $c_{\max}$ (independent of $n$)
such that
\[
0<c_{\min}\leq\sigma_{\min}(M_n),\qquad \sigma_{\max}(M_n)\leq c_{\max}<\infty\qquad\forall n.
\]
Then
\begin{equation}\label{cob}
c_{\min} \|\vec{x}_n\|_2\leq\|x_n\|_{\mathcal{X}} \leq c_{\max} \|\vec{x}_n\|_2.
\end{equation}
Thus, the norms $\|\cdot\|_{\mathcal X}$ and $\|\cdot\|_2$ are equivalent. We will therefore
simply write $\vec{x}_n$ as $x_n$. The equivalence will be explicitly used in
Section \ref{sec4}.

The solution $x_n\in{\mathcal X}_n$ of \eqref{linsys} might not be a useful approximation of
the desired solution $\widehat{x}$ of \eqref{Opgl} due to a large propagated error
stemming from the error in the available data $y_n^\delta$. We therefore would like to
determine a bound for $\|\widehat{x}-x_n\|_{\mathcal X}$. This is generally not possible
without some additional assumptions on the solution $\widehat{x}$ of \eqref{Opgl}; in
particular, it is not sufficient that $A$ and $A_n$ be close.

Let $\Omega\subseteq\mathbb{R}^N$, ${\mathcal X}=L_2(\Omega)$, and define the Sobolev
spaces ${\mathcal H}^l={\mathcal H}^l(\Omega)$ for $l\in{\mathbb R}$. Assume that
\begin{equation}\label{InvertabilityCond}
	\|Ax\|_{\mathcal Y}\sim \|x\|_{{\mathcal H}^{-l}}
\end{equation}
holds for all $x\in{\mathcal H}^{-l}$ and some $0<l<\infty$, i.e., the operator $A:
{\mathcal H}^{-l}\to {\mathcal Y}$ is continuously invertible. The theory developed by
Natterer \cite{Natterer1977} requires that \eqref{InvertabilityCond} holds for a finite
value of $l$.

%

Example 2.1
Let $\Omega\subset\mathbb R^2$, $Z=\{(\omega , t)\in \mathbb R^3: \omega\in \mathbb R^2,~
\|\omega\|=1,~t\in \mathbb R\}$. Let $\omega^\perp$ be a unit vector perpendicular to
$\omega$, and define the Radon transform
\begin{eqnarray*}
 A&:& L_2 (\Omega)\to L_2 (Z),\\
(Ax)(\omega,s)&:=& \int\limits_{\mathbb R} x(s\omega+t\omega^\perp)\, dt.
\end{eqnarray*}
Then \eqref{InvertabilityCond} holds with $l=1/2$; see \cite{Natterer:2001aa}. $~~\Box$

Example 2.2. Let
\begin{eqnarray*}
A&:& L_2 (\mathbb R^d)\to L_2 (\mathbb R^d),\\
(Ax)(\omega,s)&:=& (k\ast x)(s)=\int\limits_{\mathbb R^d} k(s-t)x(t)\, dt,
~~s\in\mathbb R^d,
\end{eqnarray*}
for some kernel function $k\in L_2(\mathbb R^d)$. If the Fourier transform $\hat k$ of $k$
satisfies
\[
  |\hat{k}(\xi)|\sim (1+|\xi|^2)^{-\beta/2},
\]
then \eqref{InvertabilityCond} holds with $l=\beta$, see, e.g., \cite{GerthRamlau}.
 $~~\Box$\\

Assume that the minimal norm solution $\widehat{x}$ of \eqref{Opgl} lives in
${\mathcal H}^k$.
Natterer \cite{Natterer1977} shows that if the operator $A$ is injective and the subspaces
${\mathcal X}_n$, $n=1,2,\ldots~$, are chosen so that an inverse estimate is fulfilled (see
\cite[eq.\,(4.1)-(4.5)]{Natterer1977} for details on the latter), then one obtains
the bound
\begin{equation}\label{errbd}
	\|\widehat{x}-x_n\|_{\mathcal X}\le
	C \left (h(n)^k \|\widehat{x}\|_{{\mathcal H}^k}+h(n)^{-l}\delta \right )
\end{equation}
for some constant $C$ that can be chosen independently of $h(n)$, $\widehat{x}$, and
$\delta$. Here, $h=h(n)>0$ is a discretization parameter that depends on the approximation
property of the subspaces ${\mathcal X}_n$, $n=1,2,\ldots~$, i.e., on how well $\widehat{x}$
can be approximated by an element in ${\mathcal X}_n$; in particular, $h\searrow 0$ as
$n\rightarrow\infty$. The parameter $\delta>0$ in \eqref{errbd} is the bound
\eqref{noise}; see \cite{Natterer1977}. An optimal dimension of the discretized problem is
given by
\begin{equation}\label{nofh}
  n\cng{\sim} h^{-1}\left(\left(\frac{\delta}
  {\|\widehat x\|_{{\mathcal H}^k}}\right)^{1/(k+l)}\right)
\end{equation}
and yields the bound
\begin{equation}\label{optimal_rate}
	\|\widehat{x}-x_n\|_{\mathcal X}\le
	C' \|\widehat x\|_{{\mathcal H}^k}^{l/(k+l)}\delta^{k/(k+l)}
\end{equation}
for some constant $C'>0$; see Natterer \cite{Natterer1977} for details. For instance,
spline and finite element approximation spaces ${\mathcal X}_n$ allow for bounds of the
type \eqref{errbd} and \eqref{optimal_rate}. Natterer \cite{Natterer1977} proposes that
the dimension $n$ of the solution subspace of the discretized problem \eqref{discOpgl} be
chosen according to \eqref{nofh}. This choice provides regularization of the operator
equation \eqref{Opgldelta} and no further regularization is necessary.

We note that the use of wavelet-based projection methods also has been investigated for
the solution of ill-posed problems. Regularization properties of wavelet methods have
been shown by Dicken and Maa\ss\ \cite{DickenMaass1996}.

Convergence rates analogous to \eqref{optimal_rate}, when $h$ is chosen according to
\eqref{nofh}, also can be established in a different setting; see Math\'e and Pereverzev
\cite{MathePerev01}. They assume that the operator $A$ is continuously invertible in
Hilbert scales (which resembles the condition \eqref{InvertabilityCond}), and show
convergence rates in a stochastic noise setting with respect to norms of the relevant
Hilbert scales; see \cite[Theorem 6.3]{MathePerev01}.

We conclude this section with a comment on condition \eqref{noise}. Let $y_n=Q_ny$ and
$y_n^\delta=Q_ny^\delta$. We will assume that
$\|y-y^\delta\|_{\mathcal Y}\approx\|y_n-y_n^\delta\|_{\mathcal Y}$. Then \eqref{noise}
translates to
\begin{equation}\label{noise2}
\|y_n-y_n^\delta\|_{\mathcal Y}\,\substack{ < \\ \approx}\, \delta.
\end{equation}
It is convenient to replace the norm in \eqref{noise2} by the Euclidean norm. This can be
achieved analogously as in the beginning of this section: Let $\{f_j\}_{j=1}^n$ form a
convenient basis for ${\mathcal Y}_n$, such as $f_j=Ae_j$, where $\{e_j\}_{j=1}^n$ is a
basis for ${\mathcal X}_n$. Represent the element $y_n\in{\mathcal Y}_n$ as
\[
y_n=\sum_{j=1}^n y_j^{(n)} f_j,
\]
and define the vector $\vec{y}_n=(y_1^{(n)},y_2^{(n)},\ldots,y_n^{(n)})^T\in{\mathbb R}^n$.
We would like to bound $\|y_n\|_{\mathcal Y}$ in terms of $\|\vec{y}_n\|_2$.
Introduce an orthonormal basis $\{\widehat{f}_j\}_{j=1}^n$ for ${\mathcal Y}_n$. There is
a nonsingular matrix $N_n\in{\mathbb R}^{n\times n}$ such that
\[
[f_1,f_2,\ldots,f_n]=[\widehat{f}_1,\widehat{f}_2,\ldots,\widehat{f}_n]N_n.
\]
It follows similarly as \eqref{ubd} that
\[
\|y_n\|_{\mathcal Y}\leq\sigma_{\max}(N_n)\|\vec{y}_n\|_2,
\]
where $\sigma_{\max}(N_n)$ denotes the largest singular value of $N_n$. We will assume
that there is an upper bound $d_{\max}$, independent of $n$, such that
\[
\sigma_{\max}(N_n)\leq d_{\max}<\infty.
\]

In computations, we will apply the discrepancy principle based on the inequality
\begin{equation}\label{discrp2}
\|y_n-y_n^\delta\|_2\leq\delta,
\end{equation}
which implies that
\[
\|y_n-y_n^\delta\|_{\mathcal Y}\,\substack{<\\ \approx}\, d_{\max}\delta.
\]

\section{Arnoldi decomposition of a matrix}\label{sec3}
Let $\Am{n}$ and $y_n^\delta$ be as in \eqref{discOpgl}, and assume that $n$ is large.
The Arnoldi process is a popular approach to reduce a large matrix to a small one by
evaluating matrix-vector products with the large matrix and applying Gram--Schmidt
orthogonalization. The small matrix, denoted by $H_{\ell+1,\ell}$ below, is an
orthogonal projection of $A_n$.
Application of $1\leq\ell\ll n$ steps of the Arnoldi process to the
matrix $\Am{n}$ with initial vector $y_n^\delta$ gives the decomposition
\begin{equation}\label{arndec}
\Am{n}V_{n,\ell} = V_{n,\ell+1} H_{\ell+1,\ell},
\end{equation}
where the columns of the matrix $V_{n,\ell+1}\in{\mathbb R}^{n,\ell+1}$ form an
orthonormal basis for the Krylov subspace
\[
{\mathcal K}_{\ell+1}(\Am{n},y_n^\delta):=
{\rm span}\{y_n^\delta,\Am{n}y_n^\delta,\ldots,\Am{n}^\ell y_n^\delta\}
\]
with respect to the inner product
\begin{equation}\label{innerprod}
\langle u, w\rangle := \frac{1}{n}\sum_{j=1}^n u_jw_j,\quad
u=(u_1,\ldots,u_n)^T,~w=(w_1,\ldots,w_n)^T\in\mathbb{R}^n
\end{equation}
and associated norm
\[
\|u\|_2:=\sqrt{\langle u, u\rangle};
\]
see, e.g., \cite{Sa}. We also will denote the spectral norm of a matrix by $\|\cdot\|_2$.
The matrix $V_{n,\ell}$ in \eqref{arndec} is made up of the first $\ell$ columns of
$V_{n,\ell+1}$, and $H_{\ell+1,\ell}\in{\mathbb R}^{\ell+1,\ell}$ is an upper Hessenberg
matrix, i.e., all entries below the subdiagonal vanish. We assume the generic situation
that the subspace ${\mathcal K}_{\ell+1}(\Am{n},y_n)$ is of dimension $\ell+1$ for all
$\ell\geq 0$, otherwise the computations simplify; see below.

\begin{algorithm}[!hbt]
\begin{algo}
\INPUT $\Am{n}\in{\mathbb R}^{n,n}$, $y_n^\delta\in{\mathbb R}^n\backslash\{0\}$, number
of steps $\ell$.
\STATE $v_1:=y_n^\delta/\|y_n^\delta\|_2$
\FOR $j=1$ \To $\ell$
\STATE $w:=\Am{n}v_j$
\FOR $k=1$ \To $j$
\STATE $h_{k,j}:=\langle v_j,w\rangle$
\STATE $w:=w - v_j h_{k,j}$
\ENDFOR
\STATE $h_{j+1,j}:= \|w\|_2$; $v_{j+1}:=w/h_{j+1,j}$
\ENDFOR
\OUTPUT Upper Hessenberg matrix $H_{\ell+1,\ell}=[h_{k,j}]\in{\mathbb R}^{\ell+1,\ell}$,
matrix
\pOUTPUT $V_{n,\ell+1}=[v_1,v_2,\ldots,v_{\ell+1}]\in{\mathbb R}^{n,\ell+1}$ with
orthonormal columns
\end{algo}
\caption{The Arnoldi process}
\label{alg:arnoldi}
\end{algorithm}

Algorithm \ref{alg:arnoldi} describes the Arnoldi process for computing the decomposition
\eqref{arndec}. The algorithm is said to break down in iteration $j$ if $h_{k+1,k}>0$ for
$1\leq k<j$, and $h_{j+1,j}=0$ in line 9. Then the decomposition \eqref{arndec} simplifies
to
\[
\Am{n}V_{n,j}=V_{n,j}H_{j,j}
\]
and the solution of \eqref{discOpgl} lives in the Krylov subspace
${\mathcal K}_j(\Am{n},y_n^\delta)$ if the matrix $H_{j,j}$ is nonsingular. This is
secured, e.g., if the matrix $\Am{n}$ is nonsingular. A discussion on how to continue the
Arnoldi process in case of breakdown when $H_{j,j}$ is singular is provided in \cite{RY1}.

We remark that the Arnoldi process simplifies to the Lanczos process when the matrix $A_n$
is symmetric; see \cite[Chapter 6]{Sa}.

\section{The Arnoldi--Tikhonov method}\label{sec4}
The results of Section \ref{sec2} suggest that the discretized system can be solved
without further regularization if the discretization is carried out on a suitably (but not
too) fine grid. However, numerical realization of regularization by discretization only
often leads to difficulties, because an appropriate value of the dimension $n$ of the
solution subspace generally is not known before the computations are begun. For instance,
when ${\mathcal X}_n$ is a finite element space, it may be necessary to determine several
discretizations and associated solutions (for different values of $n$) to find a suitable
$n$-value. We therefore prefer to first discretize the spaces $\mathcal X$ and
$\mathcal Y$ to obtain $n$-dimensional subspaces ${\mathcal X}_n$ and ${\mathcal Y}_n$,
respectively,
that allow approximation of elements of $\mathcal X$ and $\mathcal Y$ with sufficient
accuracy, and then regularize \eqref{discOpgl} by Tikhonov's method. In the remainder of
this section, we identify the spaces ${\mathcal X}_n$ and ${\mathcal Y}_n$ with
${\mathbb R}^n$ and the finite-dimensional operator $Q_nAP_n$ in \eqref{linsys} with the
matrix $A_n\in{\mathbb R}^{n,n}$ in \eqref{discOpgl}.

The solution method considered consists of three steps:
\vskip3pt
\begin{enumerate}
\item {\bf Discretization of the (infinite-dimensional) operator equation.} This requires
an estimate of the distance between the solution of the infinite-dimensional system and
the solution of its finite-dimensional approximation.
\vskip3pt
\item {\bf Definition of a regularized finite-dimensional system.} Estimate the distance
between the solution of the finite-dimensional system and its regularized version.
\vskip3pt
\item  {\bf Compute an approximate solution of the regularized solution.} Estimate the
distance between the solution of the regularized finite-dimensional system and its
computed approximation.
\end{enumerate}
\vskip3pt
The error of the computed solution is bounded by the sum of the norms of these three
errors. We will discuss each one of these errors separately.

Let the Arnoldi decomposition \eqref{arndec} be available, and introduce the approximation
\begin{equation}\label{ArnoldiApprox}
	\Aml{n}{\ell}:= \Vl{n}{\ell+1}\Hl{\ell}\Vlt{n}{\ell}
\end{equation}
of the matrix $\Am{n}$. In what follows, we need to compute an estimate for the distance
between $\Am{n}$ and $\Aml{n}{\ell}$. To this end we may compute the operator norm
$\|\Am{n}-\Aml{n}{\ell}\|_2$, which can be evaluated as the largest singular value of the
matrix $\Am{n}-\Aml{n}{\ell}$. It has recently been shown that a few of the largest
singular values of a large matrix that stems from the discretization of a linear ill-posed
problem can be computed quite inexpensively; see \cite{OR} for discussions and
illustrations. Alternatively, we may use the easily computable Frobenius norm,
\[
\|\Am{n}\|_F := \sqrt{\sum_{i,j=1}^n |a_{ij}|^2},
\]
where $A_n=[a_{ij}]_{i,j=1}^n$, and apply the bound
\[
  \|\Am{n}-\Aml{n}{\ell}\|_2\le \|\Am{n}-\Aml{n}{\ell}\|_F.
\]

Assume that
\begin{equation}\label{hbd}
	\|\Am{n}-\Aml{n}{\ell}\|_2\le h_{\ell}
\end{equation}
for some scalar $h_{\ell}>0$ and define the Tikhonov functional
\begin{equation}\label{tik1}
J_{\alpha,n,\ell}(x_n):=\|\Aml{n}{\ell}x_n-y_n^\delta\|_2^2+\alpha\|x_n\|_2^2,
\end{equation}
where $\alpha>0$ is a regularization parameter. We will solve the minimization problem
\begin{equation}\label{tiksol}
x_{\alpha,n,\ell}^\delta:=
\arg\min_{x_n\in \mathbb R^n} \left\{J_{\alpha,n,\ell}(x_n)\right\}.
\end{equation}

For comparison, we also define the Tikhonov functional $J_{\alpha,n}$ obtained by
replacing $\Aml{n}{\ell}$ in \eqref{tik1} by $A_n$, i.e.,
\[
J_{\alpha,n}(x_n):=\|A_n x_n-y_n^\delta\|_2^2+\alpha\|x_n\|_2^2
\]
and solve the minimization problem
\begin{equation}\label{tik2}
x_{\alpha,n}^\delta:=
\arg\min_{x_n\in \mathbb R^n} \left\{J_{\alpha,n}(x_n)\right\}.
\end{equation}

Let us fix $n$. We would like to choose the parameter pair $\{\ell,\alpha\}$ so that
$x_{\alpha,n,\ell}^\delta$ is an accurate approximation of the solution $\widehat x$ of
minimal norm of the operator equation \eqref{Opgl}.

The proper choice of the parameter pair $\{\ell,\alpha\}$ has been studied by Neubauer
\cite{NEUBAUER1988507}, who considers the computation of an approximate solution of an
operator equation
\[
Tx=y^\delta,\qquad T:\widetilde{\mathcal X}\to \widetilde{\mathcal Y},
\]
where $\widetilde{\mathcal X}$ and $\widetilde{\mathcal Y}$ are Hilbert spaces, by first
discretizing and then solving the discretized equation using Tikhonov regularization,
\[
x_{\alpha,\ell}^{h,\delta}:= \left (T_{h,\ell}^\ast T_{h,\ell}+\alpha I \right )^{-1}
T_{h,\ell}^\ast y_n^\delta.
\]
Here $T_{h,\ell}$ denotes a discretization and modification of $T$ (see below), and
$T_{h,\ell}^*$ is the adjoint of $T_{h,\ell}$. Neubauer \cite{NEUBAUER1988507} requires
the operator $T_{h,\ell}$ to satisfy
\begin{eqnarray*}
	\|T-T_{h,\ell}\|_2&\le & h_\ell,\\
	T_{h,\ell}&:=& \Rm{\ell}T_h,\\
	\Rm{\ell}&\to& I \mbox{~~point-wise as $\ell$ increases,}
\end{eqnarray*}
where $\Rm{\ell}$ is an orthogonal projector onto an $\ell$-dimensional subspace
${\mathcal W}_\ell\subset\widetilde{\mathcal Y}$ to be specified below. The dimension
$\ell$ is finite and typically quite small. Moreover, $T_h$ is a discretization of $T$ and
$T_{h,\ell}$ is a modification of $T_h$ determined by $\Rm{\ell}$.

In our application of the results of Neubauer \cite{NEUBAUER1988507}, we let $T:=A_n$ and
$\widetilde{\mathcal X}:=\widetilde{\mathcal Y}:={\mathbb R}^n$. Thus, we set
\begin{eqnarray}
\nonumber
T&:=& \Am{n},\\
\label{approx_qual}
T_{h,\ell}&:=& \Aml{n}{\ell},\hspace{1cm}\|\Am{n}-\Aml{n}{\ell}\|_2\le h_\ell,\\
\nonumber
{\mathcal W}_\ell&:=& \overline{\mathcal{R}(\Aml{n}{\ell})},\\
\nonumber
\Rm{\ell}&:=& P_{\overline{\mathcal{R}(\Aml{n}{\ell})}},
\end{eqnarray}
where $P_{\overline{\mathcal{R}(\Aml{n}{\ell})}}$ denotes the orthogonal projector onto
the (closure of the) range of $\Aml{n}{\ell}$. The operator $T_h$ is not important to
us; we only use $T_{h,\ell}$. We are in a position to show the following results.

\begin{proposition}
Assume that the Arnoldi process does not break down. With the operators defined as above,
we have
\begin{eqnarray}
{\mathcal W}_\ell&\subset& \overline{\mathcal{R}(\Am{n})},\label{rangecond1}\\
\mathcal{R}(\Rm{\ell}\Aml{n}{\ell})&=&{\mathcal W}_\ell,\label{rangecond2}\\
\nonumber
\|\Rm{\ell} (\Am{n}-\Aml{n}{\ell})\|_2&\le & h_\ell,\\
\label{deltabd}
\|\Rm{\ell} (y-y_n^\delta)\|_2&\le &\delta, \\
\Rm{\ell} &\to &I \mbox{ point-wise onto }\mathcal{R}(\Am{n}), \label{pointwise}
\end{eqnarray}
where the bound \eqref{deltabd} is inspired by \eqref{discrp2}.
\end{proposition}

\begin{proof}
First note that the ranges of the operators (matrices) $\Am{n}$ and $\Aml{n}{\ell}$ are
closed as they are maps between finite-dimensional spaces. It follows from \eqref{arndec}
and \eqref{ArnoldiApprox} that
\begin{equation}\label{Amnl}
\Aml{n}{\ell}=\Am{n}\Vl{n}{\ell}\Vlt{n}{\ell}
\end{equation}
and, therefore,
\begin{equation*}
{\mathcal W}_\ell=\mathcal{R}(\Aml{n}{\ell})\subset \mathcal{R}(\Am{n}),
\end{equation*}
i.e., property \eqref{rangecond1} holds. Furthermore,
\begin{equation*}
\mathcal{R}(\Rm{\ell}\Aml{n}{\ell})=
\mathcal{R}(P_{\mathcal{R}(\Aml{n}{\ell})}\Aml{n}{\ell})= \mathcal{R}(\Aml{n}{\ell})
={\mathcal W}_\ell.
\end{equation*}
This establishes \eqref{rangecond2}. Finally, we have
\begin{eqnarray*}
\|\Rm{\ell}(\Am{n}-\Aml{n}{\ell})\|_2&\le & \|\Rm{\ell}\|_2\|\Am{n}-\Aml{n}{\ell}\|_2
\stackrel{\eqref{approx_qual}}{\le} h_\ell, \\
\|\Rm{\ell} (y-y_n^\delta)\|_2&\le & \|\Rm{\ell}\|_2\|y-y_n^\delta\|_2\le \delta.
\end{eqnarray*}

It remains to show \eqref{pointwise}. According to \eqref{rangecond1}, we have
$\mathcal{R}(\Aml{n}{\ell})\subset \mathcal{R}(\Am{n})$. We will show that for every
$y_n\in \mathcal{R}(\Am{n})$, there exists an $\ell\geq 1$ such that
$y_n\in\mathcal{R}(\Aml{n}{\ell})$. Let $x_n\in \mathbb{R}^n$ and define $y_n=\Am{n}x_n$.
Note that $\Vl{n}{\ell}\Vlt{n}{\ell}$ is an orthogonal projector onto the space
$\mathcal{R}(\Vl{n}{\ell})$. Assuming that the Arnoldi process does not break down, there
is an $\ell\geq 1$ (in the worst case, $\ell=n$) such that
$x_n\in\mathcal{R}(\Vl{n}{\ell})$ and, therefore, $\Vl{n}{\ell}\Vlt{n}{\ell}x_n=x_n$. It
follows from \eqref{Amnl} that
\begin{equation*}
	y_n=\Am{n}x_n=\Am{n}\Vl{n}{\ell}\Vlt{n}{\ell}x_n=\Aml{n}{\ell}x_n,
\end{equation*}
i.e., $y_n\in \mathcal{R}(\Aml{n}{\ell})$ and, consequently, $\Rm{\ell}y_n=y_n$. This
shows the point-wise convergence of the projector $\Rm{\ell}$ to $I$ as $\ell$ increases.
\end{proof}

Thus, the requirements of Neubauer \cite[Assumption 2.3]{NEUBAUER1988507} are fulfilled,
and we get the following result from
\cite[Proposition 2.6 and Theorem 3.1]{NEUBAUER1988507}:

\begin{proposition}\label{RateTheorem}
Let $x_n$ be an approximate solution of \eqref{linsys} such that
\begin{eqnarray}\label{smoothness}
x_n &=& (\Am{n}^\ast \Am{n})^\nu v_n,\hspace{1cm}v_n\in \mathcal{N}(\Am{n})^\perp,~
\nu\in [0,1],\\
\|v_n\| &\le& \rho\hspace{2.6cm}\forall n\in\mathbb{N},\label{uniform}
\end{eqnarray}
for some constant $\rho\ge 0$ independent of $n$,
and assume that $\|y_n - \Qm{n}y^\delta\|_2\le \delta$. Let the regularization parameter
$\alpha>0$ satisfy
\begin{equation}\label{discrepancy}
\alpha^3 \left \langle \left (\Aml{n}{\ell}
\left(\Aml{n}{\ell}\right )^\ast+\alpha I\right )^{-3}
\Rm{\ell}y_n^\delta,\Rm{\ell}y_n^\delta\right\rangle = (E\,h_\ell + C\,\delta)^2,
\end{equation}
where the inner product $\langle\cdot,\cdot\rangle$ is defined by \eqref{innerprod} and
the constants $C>1$ and $E>3\|x_n\|_2$ are chosen such that
\begin{equation}\label{chkineq}
0\le E\,h_\ell +C\,\delta\le \|\Rm{\ell}y_n^\delta\|_2.
\end{equation}
Then the associated solution of \eqref{tiksol} satisfies
\begin{equation}\label{errorestimate}
\|x_{\alpha,n,\ell}^\delta - x_n\|_2 =
{\small O}\left((\delta + h_\ell)^{2\nu/(2\nu +1)}\right)+p(l,\nu)
\end{equation}
with
\begin{eqnarray*}
 p(l,\nu)&=&\left\{ \begin{array}{ll}
 0&	\mbox{ if } \nu=0\\
 \gamma_l\|(I-\Rm{\ell})z \|&\mbox{ if } \nu=\frac{1}{2},~\Am{n}^\ast z_n=
 \left (\Am{n}^\ast\Am{n}\right )^{1/2}v_n,\\
	\gamma_l^2\,\|v_n\|_2&\mbox{ if }\nu=1,\\
	(4/\pi)\gamma_l^{2\nu}\|v_n\|&\mbox{ otherwise,}
	\end{array}\right . \\
	\gamma_l&=&\|(I-\Rm{\ell})\Am{n}\|_2^2\,\|v_n\|_2
\end{eqnarray*}
Additionally, $\smallO$ has to be replaced by $\mathcal{O}$ for $\nu=1$.
\end{proposition}

~\\
\begin{remark}\label{smoothnessRemark}
The smoothness condition \eqref{smoothness} is for infinite-dimensional problems a fairly
strong restriction. In finite dimension, we observe that \eqref{smoothness} implies that
$x_n\in\mathcal{N}(\Am{n})^\perp$. Therefore, there is a unique
$v_n\in\mathcal{N}(\Am{n})^\perp$ such that $x_n=(\Am{n}^\ast \Am{n})^\nu v_n$. However,
the uniform boundedness of $\|v_n\|$, cf. inequality \eqref{uniform}, generally remains an
open problem; see Proposition \ref{sourcebound} below.
\end{remark}

\begin{remark}
The quantities in \eqref{errorestimate} may depend on $n$. Generally, $h_\ell$ does not
vary much as $\ell$ is kept fixed and $n$ is increased; see Section \ref{sec5} for
illustrations. When $n$ is fixed and $\ell$ increases, $h_\ell$ decreases. We are
interested in choosing $\ell$ large enough so that both terms in the right-hand side of
\eqref{errorestimate} are sufficiently small; see Corollary \ref{cor3} below. Also, the
condition \eqref{chkineq} requires $\ell$ to be large enough.
\end{remark}
~\\

Let us now give an example where the uniform boundedness of the source elements $v_n$,
required in Proposition \ref{RateTheorem}, can be guaranteed:\\

\begin{proposition}\label{sourcebound}
Let the conditions of Proposition \ref{RateTheorem} except for condition \eqref{uniform}
hold. Assume that $A$ is self-adjoint, fulfilling \eqref{InvertabilityCond},
$\Am{n}=\Pm{n} A\Pm{n}$, and the solution $x\in {\mathcal H}^k$ of the equation $Ax=y$
fulfills a source condition with $\nu = 1/2$ and source element
$v\in\mathcal{H}^{\tilde k}$. If $\Am{n}$ is injective, then also the solutions of the
equations $\Am{n}x_n = y_n$ fulfill a source condition with $\nu = 1/2$, and the
associated source elements $v_n$ are uniformly bounded.
\end{proposition}

\begin{proof}
For $\nu=1/2$ and self-adjoint operator $A$, the source condition transfers to
\begin{equation}\label{sourceglSA}
x=(A^\ast A)^{1/2}v = Av.
\end{equation}
As $\Am{n}$ is also self-adjoint, finite-dimensional, and injective, $x_n$ also fulfills a
source condition, see Remark \ref{smoothnessRemark},
\begin{equation*}
x_n=(\Am{n}^\ast \Am{n})^{1/2}v_n = \Am{n}v_n
\end{equation*}
with a unique $v_n$. As $A$ fulfills \eqref{InvertabilityCond}, the distance between $x$
and $x_n$ can be bounded by
\begin{equation*}
\|x-x_n\|_{\mathcal X}\le
C' \| x\|_{{\mathcal H}^k}^{l/(k+l)}\delta_n^{k/(k+l)},
\end{equation*}
see \eqref{optimal_rate}, and $\delta_n = \|y-y_n\|\searrow 0 \mbox{ as }n\to\infty$.
Using again \eqref{optimal_rate} for solving \eqref{sourceglSA} with
$\tilde\delta_n = \|x-x_n\|$, we obtain with $v\in \mathcal{H}^{\tilde k}$,
\begin{equation*}
\|v-v_n\|_{\mathcal X}\le
C' \| v\|_{{\mathcal H}^{\tilde{k}}}^{l/(\tilde{k}+l)}\delta_n^{\tilde{k}/(\tilde{k}+l)},
\end{equation*}
i.e., $v_n\to v$ and consequently $\|v_n\|$ is uniformly bounded.
\end{proof}

The best convergence rates can be achieved for $\nu = 1$:
\begin{corollary}\label{cor3}
Assume that the conditions of Proposition \ref{RateTheorem} hold, let $\nu = 1$, and let
$\hat\alpha$ solve \eqref{discrepancy}. Then for $\ell$ such that
\[
\max \{h_\ell,\|(I-\Rm{\ell})\Aml{n}{\ell}\|_2\}\sim \delta,
\]
we have
\[
\|x_{\hat\alpha,n,\ell}^\delta - x_n\|_2 = \mathcal{O}\left(\delta^{2/3}\right)
\mbox{~~as~~} \delta\searrow 0.
\]
\end{corollary}

\begin{proof}
The first term on the right-hand side of \eqref{errorestimate} behaves like
$\mathcal{O}(\delta^{3/2})$ if $h_\ell\sim \delta$. For the second term, we have
\begin{eqnarray*}
\|(I-\Rm{\ell})\Am{n}\|_2 &\le & \|(I-\Rm{\ell})\Aml{n}{\ell}\|_2+
\|(I-\Rm{\ell})(\Am{n}-\Aml{n}{\ell})\|_2\\
&\le & \|(I-\Rm{\ell})\Aml{n}{\ell}\|_2+h_\ell.
\end{eqnarray*}
Since $\Rm{\ell}\to I$ as $\ell\to n$, we can choose $\ell$ large enough such that
$\|(I-\Rm{\ell})\Aml{n}{\ell}\|_2\le\delta$ and obtain
\[
\|x_{\hat\alpha,n,\ell}^\delta - x_n\|_2 = \mathcal{O}\left(\delta^{2/3}\right)+
2\delta^2 \|v\|_2 = \mathcal{O}\left (\delta^{2/3}\right)
\]
as $\delta\searrow 0$.
\end{proof}

With the same argument we achieve optimal convergence rates for each $\nu\in (0,1)$ if
$p(l,\nu)=\smallO (\delta^{2\nu /(2\nu +1)})$, which holds for $l$ small enough.\\
Now let us further specify the orthogonal projector $\Rm{\ell}$.\\

\begin{proposition}\label{Projector}
Let $\Aml{n}{\ell}=\Vl{n}{\ell+1}\Hl{\ell}\Vlt{n}{\ell}$ be defined by
\eqref{ArnoldiApprox} and let $\Hl{\ell}=\Ul{\ell+1}\Sl{\ell}\Wlt{\ell}$ denote a
singular value decomposition, i.e., $\Ul{\ell+1}\in{\mathbb R}^{\ell+1,\ell+1}$ and
$\Wlt{\ell}\in{\mathbb R}^{\ell,\ell}$ are orthogonal matrices, whereas
$\Sl{\ell}\in{\mathbb R}^{\ell+1,\ell}$ is a diagonal matrix with nonnegative entries
arranged in nonincreasing order. In particular, all entries of the last row of $\Sl{\ell}$
vanish. Then the projector
$\Rm{\ell}:\mathbb{R}^n\to \Pm{\overline{\mathcal{R}(\Aml{n}{\ell})}}$ is given by
\begin{equation}\label{prp5}
\Rm{\ell}=\Vl{n}{\ell+1}\Ul{\ell+1}\Ikl{q}{\ell+1}\Ult{\ell+1}\Vlt{n}{\ell+1},
\end{equation}
where $\Ikl{q}{\ell+1}\in\mathbb R^{\ell+1,\ell+1}$ is defined in \eqref{Ikl} below and
$q\ge 0$ denotes the rank of the matrix $\Hl{\ell}$.
\end{proposition}

\begin{proof}
It is well known that
\[
\Pm{\overline{\mathcal{R}(\Aml{n}{\ell})}}=\Aml{n}{\ell}
\left(\Aml{n}{\ell}\right)^{\dagger}.
\]
Moreover,
\begin{equation}\label{prpeq1}
\left(\Aml{n}{\ell}\right )^\dagger=
\left(\Vl{n}{\ell+1}\Hl{\ell}\Vlt{n}{\ell}\right )^\dagger =
\Vl{n}{\ell}\Hld{\ell}\Vlt{n}{\ell+1}.
\end{equation}
The singular value decomposition of $\Hl{\ell}$ yields
\begin{equation}\label{prpeq2}
\Hld{\ell}= \Wl{\ell}\Sld{\ell}\Ult{\ell+1}.
\end{equation}
Now using \eqref{prpeq1} and \eqref{prpeq2}, we obtain
\[
\Aml{n}{\ell}\left (\Aml{n}{\ell}\right )^\dagger =
 \Vl{n}{\ell+1}\Ul{\ell+1}\Sl{\ell}\Sld{\ell}\Ult{\ell}\Vlt{n}{\ell+1}.
\]
Finally, when $\Hl{\ell}$ is of rank $q\leq\ell$, we have
\begin{equation}\label{Ikl}
\Ikl{q}{\ell+1}:=\Sl{\ell}\Sld{\ell}=\left (
	\begin{array}{cc}
	I_q & 0\\
	0& 0
\end{array}\right )\in \mathbb{R}^{\ell+1,\ell+1}
\end{equation}
with $I_q$ being the $q\times q$ identity matrix.
\end{proof}

The use of the discrepancy principle requires the solution of equation
\eqref{discrepancy}. The following result is concerned with the evaluation of the
left-hand side of this equation.

\begin{proposition}
Under the assumptions of Proposition \ref{Projector}, and with the same notation,
it holds
\begin{eqnarray}
\nonumber
    && \left \langle \left (\Aml{n}{\ell}
    \left(\Aml{n}{\ell}\right )^\ast+\alpha I\right )^{-3}
    \Rm{\ell}y_n^\delta\ ,\ \Rm{\ell}y_n^\delta\right\rangle~~~~~~  \\
\nonumber
    && \\
\label{discrepancy_simple}
    &&=\left (\Rm{\ell}y_n^\delta\right)^*\Vl{n}{\ell+1}\Ul{\ell+1}
      \left( \Ll{\ell}+\alpha \Il{\ell+1}\right )^{-3}
      \Ult{\ell+1}\Vlt{n}{\ell+1}\Rm{\ell}y_n^\delta~~~~~~\\
\label{simplified}
 &&=\left(y_n^\delta\right)^* \Vl{n}{\ell+1}\Ul{\ell+1}\Ikl{q}{\ell+1}
      \left( \Ll{\ell}+\alpha \Il{\ell+1}\right )^{-3} \Ikl{q}{\ell+1}
      \Ult{\ell+1}\Vlt{n}{\ell+1}y_n^\delta,~~~~~~
\end{eqnarray}
where $\Ll{\ell}\in\mathbb{R}^{\ell+1,\ell+1}$ is a diagonal matrix made up by the
squares of the singular values of the Hessenberg matrix
$\Hl{\ell}\in\mathbb{R}^{\ell+1,\ell}$ and with the last diagonal entry zero.
\end{proposition}

\begin{proof}
We first show \eqref{discrepancy_simple}. Using the notation of Proposition
\ref{Projector}, we obtain
\[
 \left (\Aml{n}{\ell}\right )^\ast = \Vl{n}{\ell} \Hlt{\ell}\Vlt{n}{\ell +1}
    = \Vl{n}{\ell} \Wl{\ell}\Slt{\ell}\Ult{\ell+1}\Vlt{n}{\ell+1},
\]
and taking into account that the matrices $\Ul{\ell}$ and $\Wl{\ell}$ are orthogonal, and
that the matrices $\Vl{n}{\ell}$ and $\Vl{n}{\ell+1}$ have orthonormal columns, yields
\begin{eqnarray*}
  \Aml{n}{\ell}\left (\Aml{n}{\ell}\right )^\ast &=&
    (\Vl{n}{\ell+1}\Ul{\ell+1}\Sl{\ell}\Wlt{\ell}\Vlt{n}{\ell})
  (\Vl{n}{\ell} \Wl{\ell}\Slt{\ell}\Ult{\ell+1}\Vlt{n}{\ell+1})\\
  &=& \Vl{n}{\ell+1}\Ul{\ell+1}\Sl{\ell}\Slt{\ell}\Ult{\ell+1}\Vlt{n}{\ell+1}\\
  &=&  \Vl{n}{\ell+1}\Ul{\ell+1}\Ll{\ell}\Ult{\ell+1}\Vlt{n}{\ell+1},
\end{eqnarray*}
where
\begin{equation*}
  \Ll{\ell}:= {\rm diag}\left(\sigma_1^2,\sigma_2^2,\dots,\sigma_\ell^2,0\right)
  \in\mathbb{R}^{\ell+1,\ell+1}
\end{equation*}
and $\sigma_1\ge\sigma_2\ge\ldots\ge\sigma_\ell\ge 0$ are the singular values of the
matrix $\Hl{\ell}$. We obtain
\[
\begin{array}{rcl}
 \Aml{n}{\ell}\left (\Aml{n}{\ell}\right )^\ast+\alpha I
 &=& \Vl{n}{\ell+1}\Ul{\ell+1}\left (\Ll{\ell}+\alpha \Il{\ell+1}\right)
 \Ult{\ell+1}\Vlt{n}{\ell+1} \\
 &&+\alpha(I-\Vl{n}{\ell+1}\Vlt{n}{\ell+1}).
\end{array}
\]
Since $\Vl{n}{\ell+1}\Vlt{n}{\ell+1}$ and $I-\Vl{n}{\ell+1}\Vlt{n}{\ell+1}$ are
complementary orthogonal projectors, it follows that
\[
\begin{array}{rcl}
 \left(\Aml{n}{\ell}\left (\Aml{n}{\ell}\right )^\ast+\alpha I\right)^3
 &=& \Vl{n}{\ell+1}\Ul{\ell+1}\left(\Ll{\ell}+\alpha\Il{\ell+1}\right)^3
 \Ult{\ell+1}\Vlt{n}{\ell+1} \\
 &&+\alpha^3(I-\Vl{n}{\ell+1}\Vlt{n}{\ell+1}).
\end{array}
\]
Introduce the vector
\[
z_n^\delta:=\Vl{n}{\ell+1}\Ul{\ell+1}\left(\Ll{\ell}+\alpha\Il{\ell+1}\right)^{-3}
\Ult{\ell+1}\Vlt{n}{\ell+1}R_\ell y_n^\delta.
\]
Then
\[
\left(\Aml{n}{\ell}\left (\Aml{n}{\ell}\right )^\ast+\alpha I\right)^3 z_n^\delta =
R_\ell y_n^\delta
\]
and, therefore,
\[
\left(\Aml{n}{\ell}\left(\Aml{n}{\ell}\right )^\ast+\alpha I\right)^{-3} R_\ell y_n^\delta
=z_n^\delta.
\]
This shows \eqref{discrepancy_simple}. Equation \eqref{simplified} now follows by
substituting \eqref{prp5} into \eqref{discrepancy_simple}.
\end{proof}

In actual computations, the matrix $\Hl{\ell}$ typically is small; see Section \ref{sec4}
for illustrations. The singular value decomposition of $\Hl{\ell}$ therefore is quite
inexpensive to compute and the left-hand side of \eqref{discrepancy_simple} easily can be
evaluated.

\begin{corollary}
Let the conditions in Section \ref{sec2} hold and choose $n$ according to \eqref{nofh}.
Assume that $1\le\ell\le n$ is large enough so that \eqref{discrepancy} has a solution,
which we denote by $\hat{\alpha}$. Consider the regularized solution
$x_{\hat{\alpha},n,\ell}^\delta$, defined by \eqref{tiksol} with $\alpha=\hat{\alpha}$,
an element in ${\mathcal X}_n$. Assume that the conditions of Corollary \ref{cor3} hold.
Then
\begin{equation}\label{solerrbd}
\|\widehat{x}-x_{\hat{\alpha},n,\ell}^\delta\|_{\mathcal X} \le
C'\|\widehat x\|_{{\mathcal H}^k}^{l/(k+l)}\delta^{k/(k+l)}+{\mathcal O}(\delta^{2/3})
\mbox{~~as~~} \delta\searrow 0
\end{equation}
for a suitable constant $C'>0$ with the parameter $l$ the same as in
\eqref{InvertabilityCond}.
\end{corollary}

\begin{proof}
Let $x_n\in{\mathcal X}_n$ be the minimal-norm solution \eqref{nbasis} of \eqref{linsys}
with $n$ chosen according to \eqref{nofh}. Then we obtain by the triangle inequality and
\eqref{optimal_rate} that
\begin{eqnarray*}
\|\widehat{x}-x_{\hat{\alpha},n,\ell}^\delta\|_{\mathcal X}&\le&
\|\widehat{x}-x_n\|_{\mathcal X}+\|x_n-x_{\hat{\alpha},n,\ell}^\delta\|_{\mathcal X}\\
&\le& C'\|\widehat x\|_{{\mathcal H}^k}^{l/(k+l)}\delta^{k/(k+l)}+
\|x_n-x_{\hat{\alpha},n,\ell}^\delta\|_{\mathcal X}.
\end{eqnarray*}
Now considering $x_n$ and $x_{\hat{\alpha},n,\ell}^\delta$ elements in ${\mathbb R}^n$, we
obtain from \eqref{cob} that
\[
c_{\min}\|\widehat{x}-x_{\hat{\alpha},n,\ell}^\delta\|_2\leq
\|\widehat{x}-x_{\hat{\alpha},n,\ell}^\delta\|_{\mathcal X}\leq
c_{\max}\|\widehat{x}-x_{\hat{\alpha},n,\ell}^\delta\|_2.
\]
The inequality \eqref{solerrbd} now follows from Corollary \ref{cor3}.
\end{proof}

\begin{remark}\label{rm4.10}
We conclude this section with a comment on why our analysis requires results by both
Natterer \cite{Natterer1977} and Neubauer \cite{NEUBAUER1988507}, because it may appear
more natural to choose $T_{h,\ell}=\Aml{n}{\ell}$ and apply Neubauer's result, only,
without invoking those of Natterer. Our reason for using the bounds provided by Natterer
is that in order to be able to use the results of Neubauer, without applying those of
Natterer, we need a bound for
\[
\|A-\Aml{n}{\ell}\|,
\]
where we consider $\Aml{n}{\ell}$ an operator from ${\mathcal X}$ to ${\mathcal Y}$ and
$\|\cdot\|$ denotes the appropriate operator norm. For many standard discretizations with
suitable basis functions such a bound can be determined. However, this is not the case for
the Arnoldi approximation $\Aml{n}{\ell}$, as the Arnoldi process depends on the starting
vector. Therefore, we need a discrete approximation $\Am{n}$ of $A$ so that we are able to
evaluate
\[
\|\Am{n}-\Aml{n}{\ell}\|_2
\]
numerically. Here $\Aml{n}{\ell}$ is considered a matrix. The application of the Arnoldi
process to $\Am{n}$ gives an approximation of the solution of the \it{discretized
equation}. Natterer's bounds are required to bound the distance to the solution of the
infinite-dimensional problem.
\end{remark}

\section{Computed examples}\label{sec5}
We apply the Arnoldi--Tikhonov method to a few ill-posed operator equations and illustrate
the influence of different discretizations. All computations were carried out using MATLAB
with about 15 significant decimal digits.

{\small
\begin{table}[ht]
\centering
\begin{tabular}{cccccc}
$n$ & $\ell$ & $h_\ell$ & $\alpha$ & $\|x_{\alpha,n,\ell}^\delta-x_n\|_2/\|x_n\|_2$ &
$\|x_{\alpha,n,\ell}^\delta-x_{\alpha,n}^\delta\|_2/\|x_n\|_2$ \\ \hline
$1000$ & $20$ & $1.14\cdot 10^{-1}$ & $4.90$ & $2.28\cdot 10^{-1}$ & $3.64\cdot 10^{-4}$\\
$1000$ & $30$ & $1.13\cdot 10^{-1}$ & $4.96$ & $2.28\cdot 10^{-1}$ & $3.60\cdot 10^{-4}$\\
$1000$ & $40$ & $1.12\cdot 10^{-1}$ & $4.90$ & $2.26\cdot 10^{-1}$ & $3.60\cdot 10^{-4}$\\
[1mm]
$2000$ & $20$ & $8.15\cdot 10^{-2}$ & $3.82$ & $1.95\cdot 10^{-1}$ & $3.73\cdot 10^{-4}$\\
$2000$ & $30$ & $8.13\cdot 10^{-2}$ & $3.81$ & $1.94\cdot 10^{-1}$ & $3.69\cdot 10^{-4}$\\
$2000$ & $40$ & $8.06\cdot 10^{-2}$ & $3.79$ & $1.94\cdot 10^{-1}$ & $3.67\cdot 10^{-4}$\\
[1mm]
$4000$ & $20$ & $5.78\cdot 10^{-2}$ & $2.97$ & $1.68\cdot 10^{-1}$ & $3.25\cdot 10^{-4}$\\
$4000$ & $30$ & $5.77\cdot 10^{-2}$ & $2.97$ & $1.67\cdot 10^{-1}$ & $3.24\cdot 10^{-4}$\\
$4000$ & $40$ & $5.75\cdot 10^{-2}$ & $2.96$ & $1.67\cdot 10^{-1}$ & $3.23\cdot 10^{-4}$\\
\end{tabular}
\caption{Example 5.1: The {\sf phillips} test problem. The noise level \eqref{nl} is
$1\%$.}\label{tab1}
\end{table}}

{\small
\begin{table}[ht]
\centering
\begin{tabular}{cccccc}
$n$ & $\ell$ & $h_\ell$ & $\alpha$ & $\|x_{\alpha,n,\ell}^\delta-x_n\|_2/\|x_n\|_2$ &
$\|x_{\alpha,n,\ell}^\delta-x_{\alpha,n}^\delta\|_2/\|x_n\|_2$ \\ \hline
$1000$ & $20$ & $1.14\cdot 10^{-1}$ & $4.44$ & $2.13\cdot 10^{-1}$ & $4.20\cdot 10^{-4}$\\
$1000$ & $30$ & $1.13\cdot 10^{-1}$ & $4.42$ & $2.12\cdot 10^{-1}$ & $4.16\cdot 10^{-4}$\\
$1000$ & $40$ & $1.12\cdot 10^{-1}$ & $4.36$ & $2.11\cdot 10^{-1}$ & $4.14\cdot 10^{-4}$\\
[1mm]
$2000$ & $20$ & $8.15\cdot 10^{-2}$ & $3.27$ & $1.77\cdot 10^{-1}$ & $4.28\cdot 10^{-4}$\\
$2000$ & $30$ & $8.13\cdot 10^{-2}$ & $3.26$ & $1.77\cdot 10^{-1}$ & $4.22\cdot 10^{-4}$\\
$2000$ & $40$ & $8.06\cdot 10^{-2}$ & $3.24$ & $1.76\cdot 10^{-1}$ & $4.20\cdot 10^{-4}$\\
[1mm]
$4000$ & $20$ & $5.78\cdot 10^{-2}$ & $2.42$ & $1.48\cdot 10^{-1}$ & $3.69\cdot 10^{-4}$\\
$4000$ & $30$ & $5.77\cdot 10^{-2}$ & $2.41$ & $1.47\cdot 10^{-1}$ & $3.67\cdot 10^{-4}$\\
$4000$ & $40$ & $5.75\cdot 10^{-2}$ & $2.41$ & $1.47\cdot 10^{-1}$ & $3.66\cdot 10^{-4}$\\
\end{tabular}
\caption{Example 5.1: The {\sf phillips} test problem. The noise level \eqref{nl} is
$0.1\%$.}\label{tab2}
\end{table}}

Example 5.1. Consider the Fredholm integral equation of the first kind discussed by
Phillips \cite{Ph},
\begin{equation}\label{phlps}
\int_{-6}^{6}\kappa(s,t)x(t)dt = g(t), \qquad -6\leq s\leq 6,
\end{equation}
where the solution $x(t)$, kernel $\kappa(s,t)$, and right-hand side $y(s)$ are given by
\begin{eqnarray}
\label{solut}
 x(t) &=& \left\{
    \begin{array}{ll} \vspace{0.2cm}
       1+\cos\left(\frac{\pi t}{3}\right), & \quad |t|<3, \\
       0, & \quad |t| \geq 3,
    \end{array}
 \right. \\
\nonumber
\kappa(s,t) &=& x(s-t), \\
\nonumber
y(s)&=&(6-|s|)\left(1 + \frac{1}{2}\cos\left(\frac{\pi s}{3}\right) \right) +
\frac{9}{2\pi}\sin\left(\frac{\pi |s|}{3} \right).
\end{eqnarray}
We discretize this integral equation by a Nystr\"om method based on the composite
trapezoidal rule with $n$ nodes. This yields a nonsymmetric matrix
$A_n\in{\mathbb R}^{n,n}$. The vector $x_n\in{\mathbb R}^n$ is a discretization of
the exact solution \eqref{solut}. We define the associated right-hand side $y_n=A_nx_n$,
which is assumed to be unknown. An associated contaminated right-hand side,
$y_n^\delta\in{\mathbb R}^n$, which is assumed to be known, is obtained by adding a vector
$e_n\in{\mathbb R}^n$ with normally distributed random entries with mean zero, that models
``noise,'' to $y_n$. The noise vector $e_n$ is scaled to correspond to a prescribed noise
level
\begin{equation}\label{nl}
\nu=\frac{\|e_n\|_2}{\|y_n\|_2}.
\end{equation}
We will use $\delta=\nu\|y_n\|_2$ when determining the regularization parameter
$\alpha$ by solving \eqref{discrepancy}.

Application of $\ell$ steps of the Arnoldi process to the matrix $A_n$ with initial vector
$v_1=y^\delta/\|y^\delta\|_2$ yields the decomposition \eqref{arndec}, as well as the
low-rank approximation $\Aml{n}{\ell}$ of $A_n$ defined by \eqref{ArnoldiApprox}. Table
\ref{tab1} displays the approximation error
\begin{equation}\label{hell}
h_\ell=\|A_n-\Aml{n}{\ell}\|_2;
\end{equation}
cf. \eqref{hbd}.

We determine the regularization parameter $\alpha$ by solving \eqref{discrepancy} with
$E=3\|x_n\|_2$ and $C=1$, as suggested by Proposition \ref{RateTheorem}, and then
solve the regularized problem \eqref{tiksol} with the low-rank matrix $\Aml{n}{\ell}$ for
$x_{\alpha,n,\ell}^\delta$. The inequality \eqref{chkineq} holds for all examples in this
section. Table \ref{tab1} shows the relative error
$\|x_{\alpha,n,\ell}^\delta-x_n\|_2/\|x_n\|_2$. This error depends both on the error in
$y_n^\delta$ and on the approximation error \eqref{hell}. For fixed $n$, the approximation
error \eqref{hell} is seen to decrease as $\ell$ increases in Table \ref{tab1}.

Let $x_{\alpha,n}^\delta$ denote the solution of the regularized problem \eqref{tik2} with
the matrix $A_n$. We are interested in how much the replacement of $A_n$ by the low-rank
approximation $\Aml{n}{\ell}$ affects the quality of the computed solution. Therefore, we
tabulate the normalized difference
$\|x_{\alpha,n,\ell}^\delta-x_{\alpha,n}^\delta\|_2/\|x_n\|_2$. Table \ref{tab1} shows this
difference to be much smaller than $\|x_{\alpha,n,\ell}^\delta-x_n\|_2/\|x_n\|_2$.
Hence, the use of $\Aml{n}{\ell}$ instead of $A_n$, with a fixed value of $\alpha$, does
not affect the quality of the computed solution significantly.

Table \ref{tab1} shows results for different problem sizes, $n\in\{1000,2000,4000\}$, and
noise level $1\%$. The quality of the computed solution $x_{\alpha,n,\ell}^\delta$ is seen
not to be very sensitive to the problem size $n$ or to the number of steps $\ell$ carried
out with the Arnoldi process.  For $n$ fixed, Table \ref{tab1} shows $h_\ell$ to decrease
as $\ell$ increases. Also the relative error
$\|x_{\alpha,n,\ell}^\delta-x_n\|_2/\|x_n\|_2$ can be seen to decrease slowly as $\ell$
increases. Moreover, the error decreases when $n$ increases and $\ell$ is kept fixed.

The quality of the computed solution is, of course, sensitive to the noise level. This is
illustrated by Table \ref{tab2}, which shows results for noise level $0.1\%$. The
$\alpha$-values of Table \ref{tab2} are smaller than of Table \ref{tab1}, as can be
expected. Moreover, the relative errors $\|x_{\alpha,n,\ell}^\delta-x_n\|_2/\|x_n\|$
reported in Table \ref{tab2} are smaller than the corresponding errors of Table
\ref{tab1}.

{\small
\begin{table}[ht]
\centering
\begin{tabular}{cccc}
$n$ & $\ell$ & $\sigma_1^{(\ell)}$ & $\sigma_\ell^{(\ell)}$\\
\hline
$1000$ & $20$ & $5.80$ & $2.44\cdot 10^{-4}$\\
$1000$ & $30$ & $5.80$ & $9.36\cdot 10^{-5}$\\
$1000$ & $40$ & $5.80$ & $3.37\cdot 10^{-5}$\\
[1mm]
$2000$ & $20$ & $5.80$ & $2.26\cdot 10^{-4}$\\
$2000$ & $30$ & $5.80$ & $6.44\cdot 10^{-5}$\\
$2000$ & $40$ & $5.80$ & $2.11\cdot 10^{-5}$\\
[1mm]
$4000$ & $20$ & $5.80$ & $1.99\cdot 10^{-4}$\\
$4000$ & $30$ & $5.80$ & $3.38\cdot 10^{-5}$\\
$4000$ & $40$ & $5.80$ & $1.54\cdot 10^{-5}$\\
\end{tabular}
\caption{Example 5.1: The {\sf phillips} test problem. Largest and smallest singular
values $\sigma_1^{(\ell)}\geq\ldots\geq\sigma_\ell^{(\ell)}$ of the matrices
$H_{\ell+1,\ell}$ in the definition \eqref{ArnoldiApprox} of the approximations
$\Aml{n}{\ell}$ of $A_n$ used in Table \ref{tab1}.}\label{tab3}
\end{table}}

We would like the $\ell$th singular value of $\Aml{n}{\ell}$, i.e., of $H_{\ell+1,\ell}$,
to be much smaller than the first one (the largest singular value). Then $\Aml{n}{\ell}$
captures all essential properties of $A_n$. To illustrate that this is the case, we
display in Table \ref{tab3} the largest and smallest singular values, $\sigma_1^{(\ell)}$
and $\sigma_\ell^{(\ell)}$, respectively, of the matrix $H_{\ell+1,\ell}$ in the
definition \eqref{ArnoldiApprox} of $\Aml{n}{\ell}$. The table shows singular values for
the matrices $H_{\ell+1,\ell}$ determined for Table \ref{tab1}. The size of the largest
singular value is seen to be independent of $\ell$, while the smallest singular value
decreases slowly as $\ell$ and $n$ increase. ~~~$\Box$

{\small
\begin{table}[ht]
\centering
\begin{tabular}{cccccc}
$n$ & $\ell$ & $h_\ell$ & $\alpha$ &
$\frac{\|x_{\alpha,n,\ell}^\delta-x_n\|_2}{\|x_n\|_2}$ &
$\frac{\|x_{\alpha,n,\ell}^\delta-x_{\alpha,n}^\delta\|_2}{\|x_n\|_2}$ \\[1mm] \hline
$1000$ & $20$ & $1.76\cdot 10^{-2}$ & $1.48\cdot 10^{0\phantom{-}}$ & $1.10\cdot 10^{-1}$
& $1.39\cdot 10^{-14}$\\
$1000$ & $30$ & $5.40\cdot 10^{-3}$ & $9.92\cdot 10^{-1}$ & $8.65\cdot 10^{-2}$ &
$1.87\cdot 10^{-14}$\\
$1000$ & $40$ & $2.38\cdot 10^{-3}$ & $7.94\cdot 10^{-1}$ & $8.58\cdot 10^{-2}$ &
$2.29\cdot 10^{-14}$\\
[1mm]
$2000$ & $20$ & $1.76\cdot 10^{-2}$ & $1.49\cdot 10^{0\phantom{-}}$ & $1.10\cdot 10^{-1}$
& $1.76\cdot 10^{-14}$\\
$2000$ & $30$ & $5.39\cdot 10^{-3}$ & $9.98\cdot 10^{-1}$ & $8.70\cdot 10^{-2}$ &
$2.50\cdot 10^{-14}$\\
$2000$ & $40$ & $2.21\cdot 10^{-3}$ & $8.63\cdot 10^{-1}$ & $7.99\cdot 10^{-2}$ &
$2.81\cdot 10^{-14}$\\
[1mm]
$4000$ & $20$ & $1.80\cdot 10^{-2}$ & $1.50\cdot 10^{0\phantom{-}}$ & $1.11\cdot 10^{-1}$
& $2.60\cdot 10^{-14}$\\
$4000$ & $30$ & $5.80\cdot 10^{-3}$ & $1.02\cdot 10^{0\phantom{-}}$ & $8.83\cdot 10^{-2}$
& $3.73\cdot 10^{-14}$\\
$4000$ & $40$ & $2.57\cdot 10^{-3}$ & $8.83\cdot 10^{-1}$ & $8.11\cdot 10^{-2}$
& $4.38\cdot 10^{-14}$\\
\end{tabular}
\caption{Example 5.2: The {\sf phillips} test problem. The noise level \eqref{nl} is
$1\%$.}\label{tab1b}
\end{table}}

{\small
\begin{table}[ht]
\centering
\begin{tabular}{cccccc}
$n$ & $\ell$ & $h_\ell$ & $\alpha$ &
$\frac{\|x_{\alpha,n,\ell}^\delta-x_n\|_2}{\|x_n\|_2}$ &
$\frac{\|x_{\alpha,n,\ell}^\delta-x_{\alpha,n}^\delta\|_2}{\|x_n\|_2}$ \\[1mm] \hline
$1000$ & $20$ & $1.76\cdot 10^{-2}$ & $8.73\cdot 10^{-1}$ & $8.02\cdot 10^{-2}$
& $2.16\cdot 10^{-14}$\\
$1000$ & $30$ & $5.40\cdot 10^{-3}$ & $3.55\cdot 10^{-1}$ & $4.73\cdot 10^{-2}$
& $5.05\cdot 10^{-14}$\\
$1000$ & $40$ & $2.40\cdot 10^{-3}$ & $2.21\cdot 10^{-1}$ & $3.66\cdot 10^{-2}$
& $8.52\cdot 10^{-14}$\\
[1mm]
$2000$ & $20$ & $1.75\cdot 10^{-2}$ & $8.71\cdot 10^{-1}$ & $8.01\cdot 10^{-2}$
& $2.91\cdot 10^{-14}$\\
$2000$ & $30$ & $5.39\cdot 10^{-3}$ & $3.55\cdot 10^{-1}$ & $4.73\cdot 10^{-2}$
& $6.63\cdot 10^{-14}$\\
$2000$ & $40$ & $2.22\cdot 10^{-3}$ & $2.14\cdot 10^{-1}$ & $3.60\cdot 10^{-2}$
& $1.09\cdot 10^{-13}$\\
[1mm]
$4000$ & $20$ & $1.76\cdot 10^{-2}$ & $8.73\cdot 10^{-1}$ & $8.03\cdot 10^{-2}$
& $4.45\cdot 10^{-14}$\\
$4000$ & $30$ & $5.78\cdot 10^{-3}$ & $3.73\cdot 10^{-1}$ & $4.86\cdot 10^{-2}$
& $9.99\cdot 10^{-14}$\\
$4000$ & $40$ & $2.57\cdot 10^{-3}$ & $2.31\cdot 10^{-1}$ & $3.74\cdot 10^{-2}$
& $1.63\cdot 10^{-13}$\\
\end{tabular}
\caption{Example 5.2: The {\sf phillips} test problem. The noise level \eqref{nl} is
$0.1\%$.}\label{tab2b}
\end{table}}

Example 5.2. This example also considers the integral equation \eqref{phlps}, but uses a
different discretization. The discretization is computed with the MATLAB function
{\sf phillips} from Regularization Tools by Hansen \cite{Ha}. This function uses a
Galerkin method with $n$ orthonormal box functions as test and trial functions and yields
a symmetric indefinite matrix $A_n\in{\mathbb R}^{n\times n}$. The vector
$x_n\in{\mathbb R}^n$ is a scaled discretization of the exact solution \eqref{solut}.
Since the matrix $A_n$ is symmetric, the Arnoldi process (Algorithm \ref{alg:arnoldi})
simplifies to the Lanczos process. Table \ref{tab1b} is analogous to Table \ref{tab1} and
shows results for the noise level \eqref{nl} $1\%$. Results for noise level $0.1\%$ are
displayed in Table \ref{tab2b}, which is analogous to Table \ref{tab2}. Due to the different
scaling of matrices and right-hand sides in this and the previous examples, the quantities
$h_\ell$ and $\alpha$ will differ. However, the relative errors tabulated in the last two
columns are comparable, and it is clear that the Galerkin method of the present example
furnishes more accurate approximations of $x_n$ than the Nystr\"om discretization of
Example 5.1. The exact solution $x_n$ and the computed approximation
$x_{\alpha,n,\ell}^\delta$ for $n=2000$, $\ell=30$, and $\nu=1\cdot 10^{-2}$, are shown
in Figure \ref{fig1}.
~~~$\Box$

\begin{figure}
\centering
\includegraphics[width=4in]{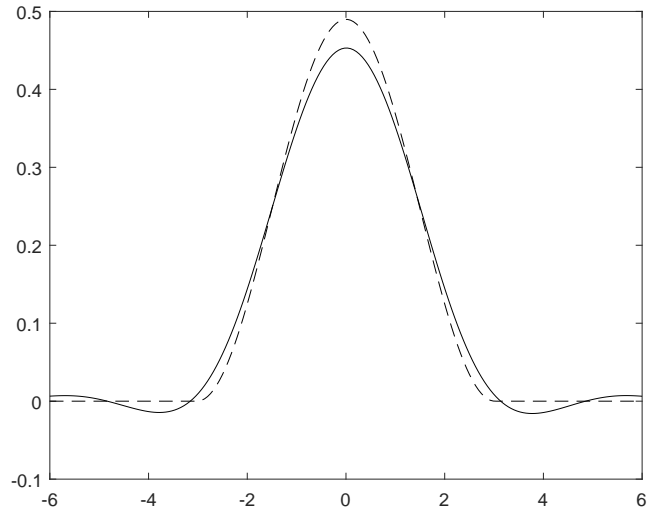}
\caption{Example 5.2: Exact solution $x_n$ (dashed curve) and computed solution
$x_{\alpha,n,\ell}^\delta$ (continuous curve) for $n=2000$, $\ell=30$, and noise
level $\nu=1\%$.}\label{fig1}
\end{figure}

Example 5.3. We turn to the Fredholm integral equation of the first kind discussed by
Baart \cite{Ba},
\[
\int_0^{\pi} \kappa(s,t)x(t)dt=g(s),\qquad 0\le s\le \frac{\pi}{2},
\]
where $\kappa(s,t)=\exp(s\cos(t))$ and $g(s)=2\sinh(s)/s$. The solution is given by
$x(t)=\sin(t)$. We discretize this integral equation by a Galerkin method using $n$
orthonormal box functions as test and trial functions. The discretization is computed
with the MATLAB function {\sf baart} from  \cite{Ha} and gives a nonsymmetric matrix
$A_n\in{\mathbb R}^{n,n}$ and a vector $x_n\in{\mathbb R}^n$ that is a scaled
discretization of the exact solution. Similarly as in Example 5.1, we define the ``unknown''
exact right-hand side by $y_n=A_nx_n$, and obtain the associated contaminated right-hand
side $y_n^\delta\in{\mathbb R}^n$, which is assumed to be known, by adding a vector
$e_n\in{\mathbb R}^n$ with normally distributed entries with zero mean to $y_n$. The
vector $e_n$ is scaled to correspond to a prescribed noise level. A few computed
results are displayed in Table \ref{tab4}. The table shows the relative error
$\|x_{\alpha,n,\ell}^\delta-x_n\|_2/\|x_n\|_2$ to be independent of $n$ for $n$ large, and
to decrease as the noise level \eqref{nl} decreases.

{\small
\begin{table}[ht]
\centering
\begin{tabular}{cccccc}
$n$ & $h_\ell$ & $\nu$ & $\alpha$ & $\frac{\|x_{\alpha,n,\ell}^\delta-x_n\|_2}{\|x_n\|_2}$ &
$\frac{\|x_{\alpha,n,\ell}^\delta-x_{\alpha,n}^\delta\|_2}{\|x_n\|_2}$ \\[1mm] \hline
$1000$ & $3.01\cdot 10^{-4}$ & $1\cdot 10^{-2}$ & $5.25\cdot 10^{-3}$ & $3.30\cdot 10^{-1}$
& $9.73\cdot 10^{-5}$ \\
$2000$ & $4.70\cdot 10^{-4}$ & $1\cdot 10^{-2}$ & $5.28\cdot 10^{-3}$ & $3.30\cdot 10^{-1}$
& $1.55\cdot 10^{-4}$ \\
$4000$ & $3.16\cdot 10^{-4}$ & $1\cdot 10^{-2}$ & $5.17\cdot 10^{-3}$ & $3.29\cdot 10^{-1}$
& $1.28\cdot 10^{-4}$ \\
[1mm]
$1000$ & $3.01\cdot 10^{-4}$ & $1\cdot 10^{-3}$ & $1.98\cdot 10^{-3}$ & $1.86\cdot 10^{-1}$
& $3.23\cdot 10^{-4}$ \\
$2000$ & $4.70\cdot 10^{-4}$ & $1\cdot 10^{-3}$ & $2.27\cdot 10^{-3}$ & $1.90\cdot 10^{-1}$
& $4.49\cdot 10^{-4}$ \\
$4000$ & $3.16\cdot 10^{-4}$ & $1\cdot 10^{-3}$ & $2.00\cdot 10^{-3}$ & $1.86\cdot 10^{-1}$
& $4.36\cdot 10^{-4}$ \\
\end{tabular}
\caption{Example 5.3: The {\sf baart} test problem for $\ell=10$, three sizes $n$, and two
noise levels $\nu$.}\label{tab4}
\end{table}}

The singular values of the matrices $A_n$, when ordered in decreasing order, decrease
rapidly with increasing index. It therefore is not meaningful to choose $\ell$ larger
than $10$. The largest singular value of all the matrices $H_{11,10}$ generated for Table
\ref{tab4} is $3.23$ and the smallest one for all matrices is about $1\cdot 10^{-13}$.

We remark that since the singular values of $A$ decrease exponentially with
their index number, the condition \eqref{InvertabilityCond} is not valid for any finite
$l$.  Nevertheless, this example illustrates that the approximation method described in
this paper also can be applied in this situation. ~~~ $\Box$

\section{Conclusion and extensions}\label{sec6}
The paper presents an analysis of the influence of discretization and truncation errors
on the computed approximate solution. These errors are caused by replacing an operator $A$
first by a large matrix $A_n$, which in turn is approximated by a matrix $\Aml{n}{\ell}$
of rank at most $\ell\ll n$. The choice of the regularization parameter in Tikhonov
regularization is discussed. Computed example illustrate the theory.

The matrix $\Aml{n}{\ell}$ is determined by the application of $\ell$ steps of the
Arnoldi process to the large matrix $A_n$. Other approaches to determine low-rank
approximations are available, such as methods based on Golub--Kahan bidiagonalization
or block Golub--Kahan bidiagonalization; see, e.g., Bentbib et al. \cite{BEGJOR} and
Gazzola et al. \cite{GNR2}. The analyses for these methods differ from the one of this
paper and are presently being pursued.

\end{document}